\documentclass[12pt]{article}

\usepackage{amssymb, amsmath, amsthm}
\usepackage{float}

\usepackage{amsfonts}
\usepackage{amscd}
\usepackage{mathbbol}
\usepackage{hyperref}
\usepackage{comment}
\usepackage{footnote}
\numberwithin{equation}{section}

\theoremstyle{plain}

\newtheorem{thm}{Theorem}[section]
\newtheorem{proposition}[thm]{Proposition}
\newtheorem{lem}[thm]{Lemma}
\newtheorem{corollary}[thm]{Corollary}

\theoremstyle{definition}

\theoremstyle{remark}

\newtheorem*{rmk}{Remark}

\newcommand{\defi}{\operatorname{def}}
\newcommand{\D}{\operatorname{D}}
\newcommand{\M}{\operatorname{M}}
\newcommand{\m}{\operatorname{m}}
\newcommand{\Ms}{\operatorname{M*}}
\newcommand{\ms}{\operatorname{m*}}
\newcommand{\Ls}{\mathcal{L}^*}

\newcommand{\floor}[1]{\lfloor #1 \rfloor}

\usepackage{tikz}
\usetikzlibrary{arrows, decorations.markings}
\usepackage{authblk}
\begin{document}

\title{Deficiency in Signed Graphs \footnote{This paper originates from a doctoral thesis written under the supervision of Thomas Zaslavsky.}}
\author{Amelia R.W. Mattern}
\affil{Binghamton University, Binghamton, New York, U.S.A.}
\maketitle

\begin{abstract}
We introduce the concept of deficiency in signed graphs. The deficiency of a coloration is the number of unused colors. We classify the deficiency of 2-chromatic graphs. There are four decision problems about the minimum and maximum deficiency of a 3-chromatic signed graph. We answer two of them with a polynomial-time algorithm for deciding the maximum deficiency of a 3-chromatic signed graph.
\end{abstract}

\section{Introduction}
In this paper we explore the novel concept of deficiency. Deficiency is a characteristic unique to signed graphs from which arises a multitude of interesting questions, many of which are yet unanswered.

A \emph{signed graph} is a graph in which every edge has an associated sign. We write a signed graph $\Sigma$ as the triple $(V, E, \sigma)$ where $V$ is the vertex set, $E$ is the edge set, and $\sigma: E \to \{+,-\}$ is the \emph{signature}. Our graphs are \emph{signed simple graphs}: no loops and no multiple edges. 

We define signed-graph coloring as in \cite{zaslavsky}, and chromatic number as in \cite{macajova}. A \emph{proper coloration} of a signed graph $\Sigma$ is a function, $\kappa : V \to \{\pm1, \pm 2, \ldots, \allowbreak \pm k, 0\},$ such that for any edge $e_{ab} \in E$, $\kappa(a) \neq \sigma(e) \kappa(b).$ The \emph{chromatic number} of $\Sigma$, written $\chi(\Sigma)$, is the size of the smallest set of colors which can be used to properly color $\Sigma.$ A graph with chromatic number $k$ is called \emph{k-chromatic}. A coloration is \emph{minimal} if it is proper and uses a set of colors of size $\chi(\Sigma).$ If $\chi = 2k,$ then a minimal \emph{color set} is $\{\pm1, \pm2, \ldots, \pm k\},$ and if $\chi = 2k+1,$ then a minimal color set is $\{\pm1, \pm2, \ldots, \pm k, 0\}.$ If $\chi = 2k+1$, there must be at least one vertex colored 0 in every minimal coloration. The \emph{deficiency} of a coloration, $\defi(\kappa)$, is the number of unused colors from the color set of $\kappa.$ The \emph{deficiency set}, $\D(\kappa)$, is the set of unused colors. 

In depictions of signed graphs we use solid lines for positive edges and dashed lines for negative edges. 

\begin{figure}[H]
\centering
\begin{tikzpicture}[
thick,
       acteur/.style={
         circle,
         fill=black,
         thick,
         inner sep=2pt,
         minimum size=0.2cm
       }, scale=.7
    ] 
\tikzset{every loop/.style={min distance=12mm,looseness=12}}
% vertices 
% Graph A
\node (A) at (0,0) [acteur, label = below: 1]{};
\node (B) at (2,0) [acteur, label = below: $-1$]{};
\node (C) at (1,1.7) [acteur, label = above: 0]{};

% Graph B
\node (a) at (4,0) [acteur, label = below: 1]{};
\node (b) at (6,0) [acteur, label = below: 0]{};
\node (c) at (5,1.7) [acteur, label = above: 1]{};

% edges
% Graph A
\draw[thick] (A) to (B);
\draw[thick, dashed] (A) -- (C); 
\draw[thick, dashed] (B) -- (C); 

% Graph B
\draw[thick] (a) to (b); 
\draw[thick, dashed] (a) to (c);
\draw[thick, dashed] (b) to (c); 
\end{tikzpicture}
\caption[A 3-chromatic graph colored in two ways]{A 3-chromatic graph colored in two ways.}
\label{deficiency example}
\end{figure}
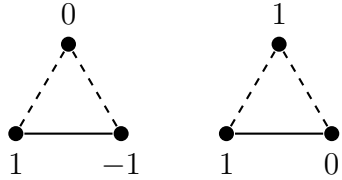

In Figure \ref{deficiency example} the coloration on the left has deficiency 0 while the coloration on the right has deficiency 1. The deficiency set of the coloration on the right is $\{-1\}.$ 

The \emph{maximum deficiency} of a graph, $\M(\Sigma) $, is $\max\{\defi(\kappa) \mid \kappa$ is a minimal proper coloration of $\Sigma \}.$ The \emph{minimum deficiency}, $\m(\Sigma)$, is $\min\{\defi(\kappa) \mid \kappa$ is a minimal proper coloration of $\Sigma\}.$ The \emph{deficiency range} of a graph, $\mathcal{L}(\Sigma)$, is $\{\defi(\kappa) \mid \kappa$ s a minimal proper coloration of $\Sigma\}.$ 

The concept of deficiency arose when considering the chromatic number of joins of signed graphs.

\begin{thm}[\cite{mattern}]
Let $\Sigma_1$ and $\Sigma_2$ be signed graphs with maximum deficiencies $M_1$ and $M_2,$ respectively. Assume that $M_1 \geq M_2.$ Then, with one exception, 
$$\chi(\Sigma_1 \vee_+ \Sigma_2) = \chi(\Sigma_1 \vee_- \Sigma_2) = \max\{\chi_1 + \chi_2 - M_1 - M_2, \chi_1\}.$$

Exception: If  $\Sigma_1$ and $\Sigma_2$ both have even chromatic number, exactly one of $M_1$ and $M_2$ is odd, and both $\Sigma_1$ and $\Sigma_2$ are exceptional graphs, then 
$$\chi(\Sigma_1 \vee_+ \Sigma_2) = \chi(\Sigma_1 \vee_- \Sigma_2) = \max\{\chi_1 + \chi_2 - M_1 - M_2 +1, \chi_1\}.$$
\end{thm}

When considering the broad question of possible deficiencies of a signed graph, it makes sense to first focus on the maximum and minimum deficiencies. The simplest difficult case is that of 3-chromatic signed graphs. There are four decision problems about the maximum and minimum deficiencies of a 3-chromatic signed graph, where deficiency can only be 0 or 1. 

\begin{enumerate}
\item Is the minimum deficiency 0?
\item Is the maximum deficiency 1?
\item Is the minimum deficiency 1?
\item Is the maximum deficiency 0?
\end{enumerate}

Questions 1 and 2 are clearly in class NP.  Questions 3 and 4, on the other hand, are neither obviously in nor not in class NP. Furthermore, if the answer to either question 3 or 4 is "yes," then questions 1 and 2 are also solved. The main result of this paper is to show that questions 2 and 4 are both in class P by providing a polynomial-time algorithm for deciding the maximum deficiency of a 3-chromatic graph.

\section{Introductory Results}
\begin{thm}
\label{2-chromatic}
Let $\Sigma$ be a 2-chromatic signed graph. The deficiencies of $\Sigma$ can be classified as follows:

Case 1: $\M(\Sigma) = 1 = \m(\Sigma)$ if and only if $\Sigma$ is connected and all negative.

Case 2: $\M(\Sigma) = 0 = \m(\Sigma)$ if and only if $\Sigma$ contains a positive edge.

Case 3: $\M(\Sigma) = 1$ and $\m(\Sigma) = 0$ if and only if $\Sigma$ is disconnected and all negative.
\end{thm}

We leave the proof of Theorem \ref{2-chromatic} to the reader. 

A subset of the vertices, $A \subseteq V$, is \emph{stable} if the subgraph induced by $A$ contains no edges.

\begin{proposition}
\label{bipartite and stable}
Let $\Sigma$ be a 3-chromatic signed graph. Then $\M(\Sigma) = 1$ if and only if $\Sigma^+$ is bipartite and there exists a bipartition of $V(\Sigma^+)$ with both parts stable in $\Sigma^+$ and one part stable in $\Sigma$.
\end{proposition}

\begin{proof}
Let $\Sigma$ be a 3-chromatic signed graph. Then there must exist at least one positive edge, meaning $\Sigma^+ \neq \emptyset.$

$(\Rightarrow)$ Suppose $\M(\Sigma) = 1.$ Then there exists a coloration of $\Sigma$ using only the colors 0 and 1. Thus, exactly two colors are used on $\Sigma^+$, meaning $\Sigma^+$ is bipartite. Furthermore, one part is colored 0, so it must be stable in $\Sigma.$ 

$(\Leftarrow)$ Suppose $\Sigma^+$ is bipartite and there exists a bipartition of $V(\Sigma^+)$ with both parts stable in $\Sigma^+$ and one part stable in $\Sigma$. Then color the part that is stable in $\Sigma$ with the color 0 and all other vertices in $\Sigma$ with the color 1. This coloration is proper since the 0-color set is stable and the 1-color set induces an all-negative subgraph. Thus $\M(\Sigma) =1.$
\end{proof}

\section{Maximum Deficiency Algorithm}
\label{s4.1}
A \emph{directed graph} is a pair $G = (V,A)$ where $V$ is the vertex set and $A$ is the multi-set of \emph{directed edges} consisting of ordered pairs of vertices. A \emph{cycle} is a directed path that starts and ends at the same vertex. The \emph{outdegree} of a vertex, the \emph{indegree} of a vertex, the minimum and maximum outdegree, and the minimum and maximum indegree are written $\delta^+(v),$ $\delta^-(v),$ $\delta^+(G),$ $\Delta^+(G)$, $\delta^-(G)$ and $\Delta^-(G)$, respectively.

 \begin{lem}
 If $G$ is a directed graph with $\delta^+(G) \geq 1$, then $G$ contains a cycle.
 \label{digraph cycles}
 \end{lem}

Before presenting the maximum deficiency algorithm, we rephrase Proposition \ref{bipartite and stable} to give an alternate way to think about maximum deficiency of a 3-chromatic signed graph. A \emph{vertex cover} of $F$ is a set of vertices of $\Sigma$ such that every edge in $F$ is incident with at least one vertex in the set. For $\alpha \in \{+,-\}$ we use the following two notations to represent specific subsets of the edge and vertex sets: $E^\alpha = \{e \in E \mid \sigma(e) = \alpha\}$ and $V^\alpha = \{v \in V \mid \text{there exists an edge of sign $\alpha$ incident with } v\}.$ 

\begin{lem}
For a 3-chromatic signed graph $\Sigma$, $\M(\Sigma) = 1$ if and only if there exists a vertex cover of $E^+$ that is stable in $\Sigma$. 
\label{vertex cover}
\end{lem}

\subsection{MaxDef}
Given a 3-chromatic, not necessarily simple, signed graph, $\Sigma,$ we provide an algorithm that decides whether the maximum deficiency of $\Sigma$ is 1 or 0. We call the algorithm MaxDef. We provide a worked example of MaxDef in section \ref{ss4.2.3}. 

The input for MaxDef is a 3-chromatic signed graph. We assume the graph is input as a pair of adjacency lists; one array of lists contains a list of the positive adjacencies for each vertex, and the other array of lists does the same for the negative adjacencies. Note that there must be at least one positive edge because the graph is 3-chromatic. The output of MaxDef is the maximum deficiency of the graph, either a 0 or a 1, along with a stable vertex cover of $E^+$ if the maximum deficiency is 1.

We build a partial stable cover throughout this process and end either in the creation of a stable cover of $E^+$ or in the non-existence of such a cover. In order to produce a stable cover, if one exists, we store recovery information about the identification of vertices. We use the symbol $\rightarrow$ to indicate replacement. For example, $A \cup B \rightarrow A$ means replace object $A$ with object $A \cup B.$ We use $N_-(v)$ to indicate the set of vertices of $\Sigma$ that are negatively adjacent to $v.$ Let $S$ be the partial stable cover, starting with $S = \emptyset.$ Let $B$ be a list of vertices that are forbidden from being in $S,$ starting with $B = \emptyset.$

{\bf Step 1:} Delete all vertices of $\Sigma$ not in $V^+.$ Every vertex incident to only negative edges does not affect the existence of a stable cover of $E^+$, and therefore, by Lemma \ref{vertex cover}, the existence of a deficiency-1 coloration. Call this set of vertices $A.$ Then let $\Sigma_{01} = \Sigma \setminus A.$

{\bf Step 2:} Let $H_1, H_2, \ldots, H_r$ be the connected components of $\Sigma^+_{01}$. For each connected component of $\Sigma^+_{01},$ check to see if it is bipartite. 

If $H_i$ is not bipartite, then $\M(\Sigma_{01}) = 0$ and thus, $\M(\Sigma) = 0$ and the algorithm ends here.

If $H_i$ is bipartite, let its unique vertex bipartition have parts $A_i$ and $B_i$. Start to create a new graph from $\Sigma_{01}$ by collapsing $A_i$ to a vertex $a_i$ and each $B_i$ to a vertex $b_i$ and eliminating multiple edges of the same sign. This makes $H_i$ into a positive edge with endpoints $a_i$ and $b_i.$ If there exists a negative edge $a_ib_i$, delete it.

If all $H_i$ are bipartite, then at the end of Step 2 we have a flattened graph, call it $\Sigma_{02}$, composed of a positive perfect matching with all other edges negative. For a matched pair $(a_i, b_i)$ we use $x_i$ to represent one vertex of the pair and $\bar x_i$ to represent the other.

\begin{lem} 
\label{1-1}
There is a one-to-one correspondence between stable covers of $E^+$ and stable covers of $E^+_{02}.$
\end{lem}

\begin{proof}
For each $i$, every stable cover of $E^+$ uses \textit{all} the vertices of either $A_i$ or $B_i.$ This corresponds to having exactly one vertex of the pair $(a_i, b_i)$ in every stable cover of $E^+_{02}$. In addition, the removal of any negative $x_i \bar x_i$ edge does not affect the 0/1 outcome of MaxDef.
\end{proof}

Steps 1 and 2 are executed only once during the algorithm. The remainder of the steps are part of a recursive process and may be executed multiple times. Let $\Sigma_1 = \Sigma_{02}.$

{\bf Step 3:} Look for an $i$ such that $x_i$ and $\bar{x}_i$ are both in $B.$ If such an $i$ exists, then $\M(\Sigma) = 0.$ If no such $i$ exists, then continue to Step 4.

{\bf Step 4:} Look for an $i$ such that $x_i$ is in $B.$ If such an $i$ exists, then $S \cup \{\bar x_i\} \rightarrow S,$ $(B \cup N_-(\bar x_i)) \setminus \{x_i\} \rightarrow B,$ and $\Sigma_1 \setminus \{x_i, \bar x_i\} \rightarrow \Sigma_1.$ Return to Step 3. If no such $i$ exists, then continue to Step 5.

{\bf Step 5:} Look for an $i$ such that $x_i$ and $\bar x_i$ both have loops. If such an $i$ exists, then $\M(\Sigma) = 0.$ If no such $i$ exists, continue to Step 6.

{\bf Step 6:} Look for an $i$ such that $x_i$ has a loop. If such an $i$ exists, then $S \cup \{\bar x_i \} \rightarrow S,$ $B \cup N_-(\bar x_i) \rightarrow B$, and $\Sigma_1 \setminus \{x_i, \bar x_i\} \rightarrow \Sigma_1.$ Return to Step 3. If no such $i$ exists, continue to Step 7.

{\bf Step 7:} Look for an $i$ such that $x_i$ is adjacent to both $x_j$ and $\bar x_j$. If such an $i$ exists, then $S \cup \{\bar x_i\} \rightarrow S,$ $B \cup N_-(\bar x_i) \rightarrow B,$ and $\Sigma_1 \setminus \{x_i, \bar x_i\} \rightarrow \Sigma_1.$ Return to Step 3. If no such $i$ exists, then continue to Step 8.

{\bf Step 8:} Look for an $i$ and $j,$ $i \neq j,$ such that $x_i$ is adjacent to $x_j$ and $\bar x_i$ is adjacent to $\bar x_j$. If such an $i$ and $j$ exist, then identify vertex $x_i$ to $\bar x_j$, and vertex $\bar x_i$ to $x_j.$ We are identifying vertices that must always be in a stable cover together. So $\Sigma_1$ (with the above stated identifications) $\rightarrow \Sigma_1.$ Delete negative edge $x_i \bar x_i$, and return to Step 3. If no such $i$ and $j$ exist, then continue to Step 9.

{\bf Step 9:} Look for a vertex of degree one. If there exists such a vertex, call it $x_i,$ then $S \cup \{x_i\} \rightarrow S,$ $B \cup N_-(x_i) \rightarrow B,$ and $\Sigma_1 \setminus \{x_i, \bar x_i\} \rightarrow \Sigma_1.$ If both $x_i$ and $\bar x_i$ are vertices of degree 1, then we only place one into $S$; it doesn't matter which one. Return to Step 3. If no such vertex exists, then continue to Step 10.

Step 9 is the only step where the vertex we add to $S$ is not a forced addition. We show that we can add it to $S$ without affecting the 0/1 outcome of MaxDef.

\begin{lem} Let $x_i$ be a vertex of degree one not in $B$. Then there exists a stable cover of $E_1^+$ extending $S$ if and only if there exists a stable cover extending $S$ which contains $x_i.$
\end{lem} 

\begin{proof} Suppose $\Sigma_1$ has a vertex of degree one that is not in $B$. Call this vertex $x_i.$ Recall that $E_1^+$ is a perfect matching. Suppose $S_1$ is a stable cover of $E_1^+$ that extends $S.$ If $S_1$ contains $x_i$, then the lemma holds. So suppose $S_1$ does not contain $x_i;$ then $S_1$ must contain $\bar x_i.$ Thus, $S_1$ is a stable cover of $E_1^+$ if and only if $S_1' = S_1 \setminus \{\bar x_1\}$ is a stable cover of $E^+$ in $\Sigma_1 \setminus \{x_i, \bar x_i\}.$ Because $x_i$ has degree one and is not in $B$, it has no neighbors in $S_1'.$ Therefore, $S_1'$ is stable in $\Sigma_1 \setminus \{x_i, \bar x_i\}$ if and only if $S_2 = S_1' \cup \{x_i\}$ is stable in $\Sigma_1.$ Moreover, since $x_i$ covers edge $x_i \bar x_i,$ $S_1'$ is a cover of $E^+$ in $\Sigma_1 \setminus \{x_i, \bar x_i\}$ if and only if $S_2$ is a cover of $E_1^+$.
\end{proof}

{\bf Step 10:} If $\Sigma_1$ is the empty graph, then $\M(\Sigma) = 1.$ Use the recovery information to produce a stable cover of $E^+$. If $\Sigma_1$ is not the empty graph, then continue to Step 11.

Upon reaching Step 11, $\Sigma_1$ is simple, is composed of a positive perfect matching with all other edges negative, contains no loops and no vertices of degree 1, and has at most one edge between every two matched pairs $(a_i, b_i)$ and $(a_j, b_j).$ 

{\bf Step 11:} Create the \emph{forcing graph} of $\Sigma_1$, called $F(\Sigma_1)$. We define $F(\Sigma_1)$ as a directed graph with $V = V_1$ and $E = \{x_ix_j \mid x_i\bar x_j \in E_1^-\}.$ Every edge in $F(\Sigma_1)$ represents a forced decision about which vertices of $\Sigma_1$ must appear together in every stable cover of $E^+_1.$ For example, if $x_ix_j \in E$, then $x_i\bar x_j \in E_1^-$. This means that if $x_i$ is in a stable cover of $E^+_1$, then $\bar x_j$ cannot be, and so $x_j$ must also be in the stable cover.

Observe that $x_ix_j \in E$ if and only if $\bar x_j\bar x_i \in E.$ Also, $\delta_+(F(\Sigma_1)) \geq 1,$ because there are no vertices of degree one in $\Sigma_1.$ Thus, there must exist a cycle in $F(\Sigma_1).$ By our first observation, there must in fact be two cycles, possibly not disjoint; if one cycle uses vertices $x_1, \ldots, x_r,$ then the other must use vertices $\bar x_1, \ldots, \bar x_r.$ Choosing a single vertex for the stable cover from a cycle forces the rest of the vertices in the cycle to also be chosen.

{\bf Step 12:} If the cycles found in Step 11 are not disjoint, then $\M(\Sigma) = 0.$ If they are, identify the vertices of $\Sigma_1$ that belong to each of the cycles from the pair found in Step 11. Eliminate multiple edges of the same sign and delete any negative $x_i \bar x_i$ edge. Then $\Sigma_1$ (with the above identifications) $\rightarrow \Sigma_1$ and return to Step 3.

This ends the MaxDef algorithm.

\begin{thm}
The MaxDef algorithm correctly decides the maximum deficiency of $\Sigma$ and is a finite process.
\end{thm}

\begin{proof}
One possible way for MaxDef to end is with $\M(\Sigma) = 0$ in Steps 2, 3, 5, and 12. The other possibility is to end with $\M(\Sigma) =1$ in Step 10. If the graph does not satisfy the ending conditions of Steps 2, 3, 5, or 12, at least one of the following Steps must be completed: 4, 6, 7, 8, 9, 10, or 12. In fact, Step 12 is not reached unless the graph does not satisfy the conditions of Steps 3--10. In the completion of Steps 4, 6, 7, 8, 9, or 12, the graph is reduced by at least one matched pair. Thus, since our graph is finite, if it never satisfies the ending conditions of Steps 2, 3, 5, or 12, it must eventually be an empty graph and satisfy the ending condition of Step 10.

Now we show that MaxDef correctly decides the maximum deficiency. We first prove if MaxDef returns 1, then $\M(\Sigma) = 1.$ 

Suppose MaxDef returns 1. Then MaxDef ends with Step 10. Let $S$ be the set of vertices produced, and let $S'$ be $S$ restricted to $\Sigma_{02}$. Observe that $S'$ is stable in $\Sigma_{02}.$ Indeed, any vertex added to the partial stable cover (through Steps 4, 6, 8, and 9) was guaranteed by Steps 3 and 4 to be allowed in the partial stable cover. Furthermore, any vertices that were identified through Steps 8 and 12 were a stable set, as guaranteed by Step 7 and the definition of the forcing graph. Also observe that $S'$ is a cover of $E^+_{02}$ since exactly one of every matched pair is present in $S'$ by definition of the algorithm. By Lemma \ref{1-1}, if $S'$ is a stable cover of $E^+_{02}$, then $S$ is a stable cover of $E^+_{01}$. Thus, $S$ is a stable cover of $E^+$ since $E^+ = E^+_{01},$ and $\Sigma_{01}$ is an induced subgraph of $\Sigma.$ Therefore, by Lemma \ref{vertex cover} $\M(\Sigma) = 1.$

Suppose MaxDef returns 0. Then the algorithm ended with Step 2, 3, 5, or 12.

Suppose MaxDef ended with Step 2. Then there exists a connected component of $\Sigma_{01}^+$ that is not bipartite. Thus, it is also a connected component of $\Sigma^+$ that is not bipartite. By Lemma \ref{bipartite and stable}, $\M(\Sigma) = 0.$

Suppose MaxDef ended with Step 3. Then at some point in the algorithm, there exists a matched pair of vertices that were both in $B$. Since all additions to the partial stable set are either forced or of degree one---and therefore cannot be a neighbor of either vertices in the matched pair---this means a stable cover of $E^+_1$ does not exist. Thus, a stable cover of $E^+_{01}$, and therefore of $E^+,$ does not exist.

Suppose MaxDef ended with Step 5. Then after completing either Step 2 or Step 12, 
there exists a matched pair of vertices that both have loops. Supposed the matched pair is $(a_1, b_1).$ Loops in $\Sigma_1$ directly after Step 2 come from either pre-existing loops, or internal edges in $A_1$ and $B_1$. Loops in $\Sigma_1$ after completing Step 12 come from sets of vertices with internal negative edges being identified. In either case, choosing either $a_1$ or $b_1$ results in choosing an unstable set of vertices for the stable cover.

Finally, suppose MaxDef ended with Step 12. Then there existed a pair of cycles in the forcing graph that were not disjoint. Thus, choosing a single vertex from either cycle forces the rest of the vertices in both cycles to be chosen. Therefore, a matched pair, $x_i$ and $\bar x_i$, must both be chosen and it is impossible to create a stable cover.
\end{proof}

\subsection{Example of MaxDef}
\label{ss4.2.3}
Let the following graph be our original $\Sigma.$ We start with $S = \emptyset,$ $B = \emptyset,$ and a recovery array $R = \emptyset.$

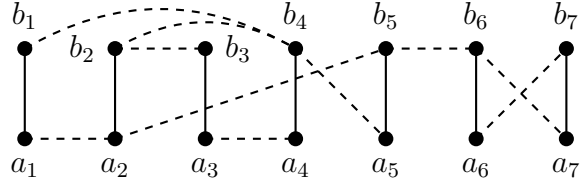
\begin{figure}[H]
\centering
\begin{tikzpicture}[
thick,
       acteur/.style={
         circle,
         fill=black,
         thick,
         inner sep=2pt,
         minimum size=0.2cm
       }, scale=.6
    ] 
\tikzset{every loop/.style={min distance=12mm,looseness=12}}
% vertices 
\node at (0,0) [acteur, label = below: $a_1$]{};
\node at (0,2) [acteur, label = above: $b_1$]{};

\node at (2,0) [acteur, label = below: $a_2$]{};
\node at (2,2) [acteur, label = left: $b_2$]{};

\node at (4,0) [acteur, label = below: $a_3$]{};
\node at (4,2) [acteur, label = right: $b_3$]{};

\node at (6,0) [acteur, label = below: $a_4$]{};
\node at (6,2) [acteur, label = above: $b_4$]{};

\node at (8,0) [acteur, label = below: $a_5$]{};
\node at (8,2) [acteur, label = above: $b_5$]{};

\node at (10,0) [acteur, label = below: $a_6$]{};
\node at (10,2) [acteur, label = above: $b_6$]{};

\node at (12,0) [acteur, label = below: $a_7$]{};
\node at (12,2) [acteur, label = above: $b_7$]{};

% edges
\draw[thick] (0,0) -- (0,2); 
\draw[thick] (2,0) -- (2,2); 
\draw[thick] (4,0) -- (4,2); 
\draw[thick] (6,0) -- (6,2); 
\draw[thick] (8,0) -- (8,2); 
\draw[thick] (10,0) -- (10,2); 
\draw[thick] (12,0) -- (12,2); 

\draw[thick, dashed] (0,0) -- (2,0);
\draw[thick, dashed] (0,2) to[bend left] (6,2);
\draw[thick, dashed] (2,0) -- (8,2);
\draw[thick, dashed] (2,2) -- (4,2);
\draw[thick, dashed] (2,2) to[bend left] (6,2);
\draw[thick, dashed] (4,0) -- (6,0);
\draw[thick, dashed] (6,2) -- (8,0);
\draw[thick, dashed] (8,2) -- (10,2);
\draw[thick, dashed] (10,0) -- (12,2);
\draw[thick, dashed] (10,2) -- (12,0);

\end{tikzpicture}
\caption[$\Sigma$ is the starting graph for our example]{$\Sigma$ is a 3-chromatic graph.}
\label{original}
\end{figure}

Because $\Sigma$ is composed of a positive perfect matching with all other edges negative we skip Steps 1 and 2. At this point $\Sigma_1 = \Sigma$, $S = \emptyset, B = \emptyset,$ and $R$ has an empty list for each vertex.

{\bf Step 8:} Using Step 8 we can identify vertices $b_6$ and $b_{7}$, calling this vertex $b_6$ and vertices $a_6$ and $a_{7},$ calling this vertex $a_6.$ We update $R$ to get the table below.

{\centering
\begin{tabular}{| c | c | c | c | c | c | c | c | c | c | c | c | c | c |}
\hline
$a_1$ & $b_1$ & $a_2$ & $b_2$ & $a_3$ & $b_3$ & $a_4$ & $b_4$ & $a_5$ & $b_5$ &
$a_6$ & $b_6$ \\ \hline
$\{\}$ & $\{\}$ & $\{\}$ & $\{\}$ & $\{\}$ & $\{\}$ & $\{\}$ & $\{\}$ & $\{\}$ & $\{\}$ & $\{a_7\}$ & $\{b_7\}$ \\ \hline
\end{tabular} \par
}

\begin{figure}[H]
\centering
\begin{tikzpicture}[
thick,
       acteur/.style={
         circle,
         fill=black,
         thick,
         inner sep=2pt,
         minimum size=0.2cm
       }, scale=.6
    ] 
\tikzset{every loop/.style={min distance=12mm,looseness=12}}
% vertices 
\node at (0,0) [acteur, label = below: $a_1$]{};
\node at (0,2) [acteur, label = above: $b_1$]{};

\node at (2,0) [acteur, label = below: $a_2$]{};
\node at (2,2) [acteur, label = left: $b_2$]{};

\node at (4,0) [acteur, label = below: $a_3$]{};
\node at (4,2) [acteur, label = right: $b_3$]{};

\node at (6,0) [acteur, label = below: $a_4$]{};
\node at (6,2) [acteur, label = above: $b_4$]{};

\node at (8,0) [acteur, label = below: $a_5$]{};
\node at (8,2) [acteur, label = above: $b_5$]{};

\node at (10,0) [acteur, label = below: $a_6$]{};
\node at (10,2) [acteur, label = above: $b_6$]{};

% edges
\draw[thick] (0,0) -- (0,2); 
\draw[thick] (2,0) -- (2,2); 
\draw[thick] (4,0) -- (4,2); 
\draw[thick] (6,0) -- (6,2); 
\draw[thick] (8,0) -- (8,2); 
\draw[thick] (10,0) -- (10,2); 

\draw[thick, dashed] (0,0) -- (2,0);
\draw[thick, dashed] (0,2) to[bend left] (6,2);
\draw[thick, dashed] (2,0) -- (8,2);
\draw[thick, dashed] (2,2) -- (4,2);
\draw[thick, dashed] (2,2) to[bend left] (6,2);
\draw[thick, dashed] (4,0) -- (6,0);
\draw[thick, dashed] (6,2) -- (8,0);
\draw[thick, dashed] (8,2) -- (10,2);

\end{tikzpicture}
\caption[$\Sigma_1$ is the starting graph for our example]{$\Sigma_1$ after executing Step 8.}
\label{original}
\end{figure}
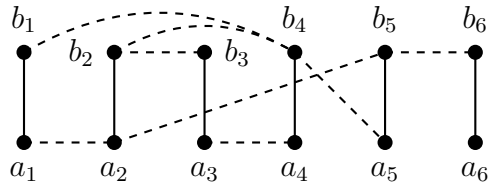

{\bf Step 9:} Vertex $a_6$ has degree one, so we update $S$ to be $S = \{a_6\}.$

\begin{figure}[H]
\centering
\begin{tikzpicture}[
thick,
       acteur/.style={
         circle,
         fill=black,
         thick,
         inner sep=2pt,
         minimum size=0.2cm
       }, scale=. 6
    ] 
\tikzset{every loop/.style={min distance=12mm,looseness=12}}
% vertices 
\node at (0,0) [acteur, label = below: $a_1$]{};
\node at (0,2) [acteur, label = above: $b_1$]{};

\node at (2,0) [acteur, label = below: $a_2$]{};
\node at (2,2) [acteur, label = left: $b_2$]{};

\node at (4,0) [acteur, label = below: $a_3$]{};
\node at (4,2) [acteur, label = right: $b_3$]{};

\node at (6,0) [acteur, label = below: $a_4$]{};
\node at (6,2) [acteur, label = above: $b_4$]{};

\node at (8,0) [acteur, label = below: $a_5$]{};
\node at (8,2) [acteur, label = above: $b_5$]{};

% edges
\draw[thick] (0,0) -- (0,2); 
\draw[thick] (2,0) -- (2,2); 
\draw[thick] (4,0) -- (4,2); 
\draw[thick] (6,0) -- (6,2); 
\draw[thick] (8,0) -- (8,2); 

\draw[thick, dashed] (0,0) -- (2,0);
\draw[thick, dashed] (0,2) to[bend left] (6,2);
\draw[thick, dashed] (2,0) -- (8,2);
\draw[thick, dashed] (2,2) -- (4,2);
\draw[thick, dashed] (2,2) to[bend left] (6,2);
\draw[thick, dashed] (4,0) -- (6,0);
\draw[thick, dashed] (6,2) -- (8,0);

\end{tikzpicture}
\caption[$\Sigma_1$ is the starting graph for our example]{$\Sigma_1$ after executing Step 9.}
\label{original}
\end{figure}
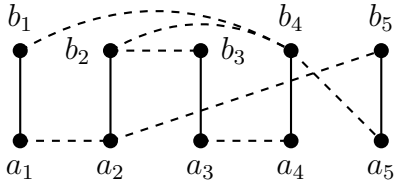

{\bf Step 11:} Consider the forcing graph of $\Sigma_1.$

\begin{figure}[H]
\centering
\begin{tikzpicture}[
thick,
       acteur/.style={
         circle,
         fill=black,
         thick,
         inner sep=2pt,
         minimum size=0.2cm
       }, scale=.6
    ] 
\tikzset{every loop/.style={min distance=12mm,looseness=12}}
\tikzset{edger/.style={decoration={
  markings,
  mark=at position #1 with {\arrow{angle 45}}},postaction={decorate}, color = red}}\tikzset{edgeb/.style={decoration={
  markings,
  mark=at position #1 with {\arrow{angle 45}}},postaction={decorate}, color = blue}}
\tikzset{edge/.style={decoration={
  markings,
  mark=at position #1 with {\arrow{angle 45}}},postaction={decorate}}}
% vertices 
\node at (0,0) [acteur, label = right: $a_1$]{};
\node at (0,2) [acteur, label = left: $b_1$]{};

\node at (2,0) [acteur, label = below: $a_2$]{};
\node at (2,2) [acteur, label = left: $b_2$]{};

\node at (4,0) [acteur, label = right: $a_3$]{};
\node at (4,2) [acteur, label = right: $b_3$]{};

\node at (6,0) [acteur, label = above: $a_4$]{};
\node at (6,2) [acteur, label = below: $b_4$]{};

\node at (8,0) [acteur, label = below: $a_5$]{};
\node at (8,2) [acteur, label = above: $b_5$]{};

% edges
\draw [edge=.8] (0,0) to (2,2);
\draw [edge=.8] (2,0) to (0,2);
\draw [edge=.7] (6,2) to[bend right = 90, looseness = 1.3] (0,0);
\draw [edge=.7] (0,2) to[bend right = 90, looseness = 1.3] (6,0);
\draw [edge=.85] (2,2) to (4,0);
\draw [edge=.8] (4,2) to (2,0);
\draw [edge=.5] (2,0) to[bend right] (8,0);
\draw [edge=.5] (8,2) to[bend right] (2,2);
\draw [edge=.8] (2,2) to (6,0);
\draw [edge=.45] (6,2) to (2,0);
\draw [edge=.7] (4,0) to (6,2);
\draw [edge=.45] (6,0) to (4,2);
\draw [edge=.6] (6,2) to (8,2);
\draw [edge=.6] (8,0) to (6,0);
\end{tikzpicture}
\caption[The forcing graph of $\Sigma_1$]{The forcing graph of $\Sigma_1$.}
\label{forcing graph 1}
\end{figure}
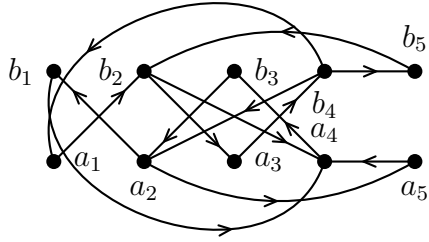

{\bf Step 12:} Identify the vertices, $a_1, b_2, a_3,$ and $b_4,$ and call this vertex $a_1.$ Identify the vertices, $a_4, b_3, a_2,$ and $b_1,$ and call this vertex $b_1.$ Our recovery array is updated to be the following:

{\centering
\begin{tabular}{| c | c | c | c | c | c |}
\hline
$a_1$ & $b_1$ & $a_5$ & $b_5$ & $a_6$ & $b_6$ \\ \hline
$\{b_2,a_3,b_4\}$ & $\{a_2,b_3,a_4\}$ & $\{\}$ & $\{\}$ & $\{a_7\}$ & $\{b_7\}$ \\ \hline
\end{tabular} \par}

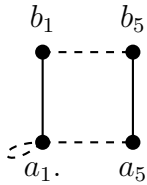
\begin{figure}[H]
\centering
\begin{tikzpicture}[
thick,
       acteur/.style={
         circle,
         fill=black,
         thick,
         inner sep=2pt,
         minimum size=0.2cm
       }, scale=.6
    ] 
\tikzset{every loop/.style={min distance=12mm,looseness=12}}
% vertices 
\node at (6,0) [acteur, label = below: $a_1.$]{};
\node at (6,2) [acteur, label = above: $b_1$]{};

\node at (8,0) [acteur, label = below: $a_5$]{};
\node at (8,2) [acteur, label = above: $b_5$]{};

% edges
\draw[thick] (6,0) -- (6,2); 
\draw[thick] (8,0) -- (8,2); 

\draw[thick, dashed] (6,0) to[in= 180, out =220, loop] (6,0);
\draw[thick, dashed] (6,0) -- (8,0);
\draw[thick, dashed] (6,2) -- (8,2);

\end{tikzpicture}
\caption[After identifying the cycles of the forcing graph]{$\Sigma_1$ is the result of identifying the vertices in each of the two cycles.}
\label{sigma 1}
\end{figure}

{\bf Step 6:} Vertex $b_1$ must be in every stable cover since its matched partner has a loop. Now $S = \{a_6, b_1\}$ and $B = \{b_5\}.$

\begin{figure}[H]
\centering
\begin{tikzpicture}[
thick,
       acteur/.style={
         circle,
         fill=black,
         thick,
         inner sep=2pt,
         minimum size=0.2cm
       }, scale=.6
    ] 
\tikzset{every loop/.style={min distance=12mm,looseness=12}}
% vertices 
\node at (8,0) [acteur, label = below: $a_5$]{};
\node at (8,2) [acteur, label = above: $b_5$]{};

% edges
\draw[thick] (8,0) -- (8,2); 

\end{tikzpicture}
\caption[Step 6 example]{$\Sigma_1$ after removing vertices $a_1$ and $b_1$ with Step 6.}
\label{sigma 1.1}
\end{figure}
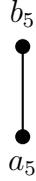

{\bf Step 4:} Vertex $b_5$ is in $B$. Thus, by Step 4, vertex $a_5$ must be in $S.$ So $S = \{a_6, b_1, a_5\}$ and $B = \emptyset.$

{\bf Step 10:} Our graph has no more vertices. Thus, $\M(\Sigma) = 1$. If we use $R$ to expand the vertices of $S$ into their original pre-identified vertices we get $S = \{a_6, b_1, a_5\} = \{a_6, a_7, a_4, b_3, a_2, b_1, a_5\}.$ Figure \ref{stable cover} shows the resulting stable cover of $E^+$.

\begin{figure}[H]
\centering
\begin{tikzpicture}[
thick,
       acteur/.style={
         circle,
         fill=black,
         thick,
         inner sep=2pt,
         minimum size=0.2cm
       }, 
       sq/.style={
         rectangle,
         fill=black,
         thick,
         inner sep=2pt,
         minimum size=0.2cm
       }, scale=.6
    ] 
\tikzset{every loop/.style={min distance=12mm,looseness=12}}
% vertices 
\node at (0,0) [acteur, label = below: $a_1$]{};
\node at (0,2) [sq, label = above: $b_1$]{};

\node at (2,0) [sq, label = below: $a_2$]{};
\node at (2,2) [acteur, label = left: $b_2$]{};

\node at (4,0) [acteur, label = below: $a_3$]{};
\node at (4,2) [sq, label = right: $b_3$]{};

\node at (6,0) [sq, label = below: $a_4$]{};
\node at (6,2) [acteur, label = above: $b_4$]{};

\node at (8,0) [sq, label = below: $a_5$]{};
\node at (8,2) [acteur, label = above: $b_5$]{};

\node at (10,0) [sq, label = below: $a_6$]{};
\node at (10,2) [acteur, label = above: $b_6$]{};

\node at (12,0) [sq, label = below: $a_7$]{};
\node at (12,2) [acteur, label = above: $b_7$]{};

% edges
\draw[thick] (0,0) -- (0,2); 
\draw[thick] (2,0) -- (2,2); 
\draw[thick] (4,0) -- (4,2); 
\draw[thick] (6,0) -- (6,2); 
\draw[thick] (8,0) -- (8,2); 
\draw[thick] (10,0) -- (10,2); 
\draw[thick] (12,0) -- (12,2); 

\draw[thick, dashed] (0,0) -- (2,0);
\draw[thick, dashed] (0,2) to[bend left] (6,2);
\draw[thick, dashed] (2,0) -- (8,2);
\draw[thick, dashed] (2,2) -- (4,2);
\draw[thick, dashed] (2,2) to[bend left] (6,2);
\draw[thick, dashed] (4,0) -- (6,0);
\draw[thick, dashed] (6,2) -- (8,0);
\draw[thick, dashed] (8,2) -- (10,2);
\draw[thick, dashed] (10,0) -- (12,2);
\draw[thick, dashed] (10,2) -- (12,0);

\end{tikzpicture}
\caption[The resulting stable cover]{The square vertices of $\Sigma$ make up the stable cover $\{b_1, a_2, b_3, a_4, a_5, a_6, a_7\}$ of $E^+$.}
\label{stable cover}
\end{figure}
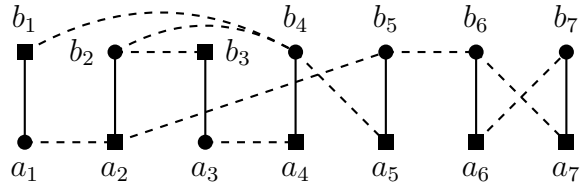

\section{Complexity}

\subsection{Data Format}
Our algorithm requires five different objects to be stored. First, we store the signed graph using adjacency lists. The adjacency lists will be two arrays of lists, one for positive edges and one for negative edges. Since each edge of $\Sigma$ appears twice in these arrays, the size of these adjacency lists is at its largest $2E.$ After Step 2, we no longer need the positive adjacency lists and stop updating them.

Second, we need to store the current forcing graph adjacency lists. Every negative edge in our signed graph gives rise to two directed edges. Therefore the size of these adjacency lists is at its largest $2E^-.$

Third, we need to store the partial stable set $S$. This is stored simply as a list of size at most $\frac{1}{2}V.$

Fourth, we need a list of vertices forbidden from the partial stable set $S.$ This is stored simply as a list of size at most $V-1$.

Finally, we store recovery information in order to produce a stable cover of the original graph if the maximum deficiency is 1. The recovery information is stored as an array of lists, one list for each group of identified vertices, and has size at most $V$.

\subsection{Time Complexity}

\begin{thm}
\label{time complexity}
If $\Sigma$ is without multiple edges of the same sign, then MaxDef has running time $\mathcal O(V^5).$
\end{thm}

\subsubsection{Updating Lists} 
Suppose vertex $x_i$ is added to the partial stable set. Then the sets $S$ and $B$ and the adjacency lists of our current version of $\Sigma$ need to be updated. This happens in several steps of the algorithm, so we calculate the time complexity of updating these three objects separately. 

In order to update $S,$ we simply append element $x_i$ to $S.$ This takes constant time. 

We update the adjacency lists and $B$ together. We look through the negative adjacency lists for both $x_i$ and $\bar x_i.$ Looking through the lists takes time $\mathcal O(E^-).$  We delete every occurrence of $\bar x_i.$ We need to delete vertex $\bar x_i$ at most $V-1$ times so deleting $\bar x_i$ from the adjacency lists takes time $\mathcal O(V).$ We delete every occurrence of $x_i$ and add the vertices in whose lists $x_i$ appeared to $B$. Since we need to delete $x_i$ at most $V-1$ times, the time complexity of deleting and updating $B$ is $\mathcal O(V).$ 

Finally, we need to possibly delete vertex $\bar x_i$ from $B.$ This requires looking through $B$, which takes time $\mathcal O(V),$ and then removing $\bar x_i$ if it appears, which takes constant time. 

Therefore, the total time complexity for updating all three objects is $\mathcal O(V + E^-).$

\subsubsection{Steps 1 Through 12}
{\bf Step 1:} We find vertices that are not in $V^+$ by looking for empty lists in the positive edge adjacency lists. Finding these vertices takes time $\mathcal O(V).$ 

In order to delete vertices from $\Sigma,$ we first remove the lists associated with those vertices. Deleting a list takes constant time and we delete at most $V-2$ lists; thus, deleting the full lists takes time $\mathcal O(V).$ We then delete any appearance of the vertices themselves in an adjacency list. Since the vertices being deleted only have negative edges, and there are at worst $V-2$ vertices being deleted, the time complexity is $\mathcal O(VE^-).$ Therefore the total time complexity for Step 1 is $\mathcal O(VE^-).$

{\bf Step 2:} In order to decide whether a given $H_i$ is bipartite, we can employ a depth-first search (DFS) algorithm. The DFS algorithm assigns a color to a vertex that is different than the color of its parent in the depth-first search tree. For an $H_i$ with $s$ vertices and $t$ edges, this algorithm takes time $\mathcal O(s + t),$ \cite{Cormen}. Therefore, checking all of the $H_i$ gives us a total time complexity of $\mathcal O(V + E^+).$ 

To collapse an $H_i$, we use the bipartition $A_i,$ $B_i$ from the DFS algorithm. Suppose $A_i$ is composed of vertices $v_1, \ldots, v_r$ and $B_i$ is composed of vertices $v_{r+1}, \ldots, v_{r+s}.$ Then we create a list for $a_i$ and a list for $b_i$ by moving every element in the negative adjacency lists of vertices $v_1, \ldots, v_r$ to the negative adjacency list for $a_i$ and every element in the negative adjacency lists of vertices $v_{r+1}, \ldots, v_{r+s}$ to the negative adjacency list for $b_i.$ Moving an element takes constant time, and we do it at most $2E^-$ times. 

We delete the negative edge $a_ib_i,$ if it exists, and remove duplicate elements from the $a_i$ and $b_i$ lists. This requires looking though the lists at most $2E^-$ times, once for each element of the lists, for a total time complexity of $\mathcal O((E^-)^2).$ 

Finally, we initialize the recovery array with enough positions for each $a_i$ and $b_i$, then add the lists for two new elements, one for vertex $a_i$ and one for vertex $b_i.$ Because each $H_i$ results in two lists for the recovery array, we put the list of vertices from $A_i$ in position $2i-1$ and the list of vertices from $B_i$ in position $2i.$ This takes constant time.

Thus, collapsing one $H_i$ takes time $\mathcal O((E^-)^2).$ We must do this process at most $V/2$ times, as each $H_i$ has at least 2 vertices. Thus, the entire collapsing process takes time $\mathcal O(V(E^-)^2).$ Therefore, the total time complexity of Step 2 is $\mathcal O(V(E^-)^2 + V + E^+).$

{\bf Step 3:} We look through $B$ once for each element of $B$. Thus, Step 3 takes time $\mathcal O(V^2).$

{\bf Step 4:} We check to see if $B$ is empty; this takes constant time. If $B$ is not empty, we add a new vertex to $S$ and update the various lists. As shown above, updating the lists is of time complexity $\mathcal O(V + E^-).$ Therefore, Step 4 takes time $\mathcal O(V + E^-).$

{\bf Step 5:} We look through the negative adjacency lists to find a matched pair of lists that both contain their own vertex. This takes time $\mathcal O(E^-).$

{\bf Step 6:} We look through the negative adjacency lists to find a list that contains its own vertex. This has time complexity $\mathcal O(E^-).$ If we find such a vertex, we must update the lists; this takes time $\mathcal O(V + E^-).$ Therefore, the total time complexity of Step 6 is $\mathcal O(V + E^-).$

{\bf Step 7:} We look through each negative adjacency list to see if it contains an $x_i, \bar x_i$ pair. This takes time $\mathcal O(V^2)$. Since we must do this for each vertex, the total time complexity for checking the lists is $\mathcal O(V^3).$ Then we must update the various lists, taking time $\mathcal(V + E^-).$ Therefore, the total time complexity of Step 7 is $\mathcal O(V^3 + E^-).$

{\bf Step 8:} We look through the negative adjacency lists once for each element of the negative adjacency lists. Thus, the time complexity of finding such vertices is $\mathcal O((E^-)^2).$ 

If we find such vertices, we must identify them. Suppose we are identifying vertices $x_i$ and $\bar x_j$ and vertices $\bar x_i$ and $x_j.$ We describe the process for identifying $x_i$ and $\bar x_j.$ Without loss of generality, assume $i < j.$ 

First, we combine the negative adjacency lists for $x_i$ and $\bar x_j.$ Concatenating the two lists takes time at most $\mathcal O(V)$. We then need to delete edge $x_i \bar x_i$ and remove duplicate elements from the newly combined lists, which as shown in Step 2 of this proof takes time $\mathcal O((E^-)^2).$

Second, we change every occurrence of $\bar x_j$ in an adjacency list to $x_i.$ Changing the element takes constant time, and looking through the adjacency lists has time complexity $\mathcal O(E^-).$

Finally, we update the recovery information. This requires concatenating the $\bar x_j$ list to the $x_i$ list, and removing the $\bar x_j$ list. Concatenating takes time at most $\mathcal O(V)$, and deleting the $\bar x_i$ list takes constant time.

Therefore, identifying vertex $x_i$ to vertex $\bar x_j$ has total time complexity $\mathcal O((V + E^-)^2).$ We must do this process twice, once for each identification. Therefore, Step 8 has total time complexity of $\mathcal O(V + (E^-)^2).$

{\bf Step 9:} This step again requires looking through the negative adjacency lists and then updating. Therefore, the time complexity of Step 9 is $\mathcal O(V + E^-).$

{\bf Step 10:} This step requires looking through the negative adjacency lists, so takes time $\mathcal O(E^-).$ If the current version of $\Sigma$ is empty, we use the recovery information to produce a stable cover of the original graph. This requires looking through $S,$ and then for each element of $S$, looking through the recovery array. Since $S$ is size at most $\frac{1}{2}V$ and the recovery array is size at most $V,$ producing a stable cover takes time $\mathcal O(V^2).$ Therefore, Step 10 takes time $\mathcal O(V^2 + E^-).$

{\bf Step 11:} To create the forcing graph, we must look through the negative adjacency lists and create a new set of adjacency lists. For every endpoint $x_i$ in the adjacency list of vertex $v_1$ we add the vertex $\bar x_i$ to the adjacency list of $v_1$ in the forcing graph. The addition of a new element to the forcing graph adjacency lists takes constant time, and there are $2E^-$ elements of the original adjacency lists. Therefore, Step 11 takes time $\mathcal O(E^-).$

{\bf Step 12:} For this step, we first find a cycle---recall the cycles come in pairs, so we need only look for one cycle. This can be done with a DFS algorithm, so takes time $\mathcal O(V + E^-)$ \cite{Cormen}. Let the cycles found be $C_1$ and $C_2.$

We then check to see whether $C_1$ and $C_2$ are disjoint. This requires reading through the list of vertices in $C_2$, once for each vertex in $C_1.$ This takes time $\mathcal O(V^2).$

If the cycles are disjoint, we identify the vertices of each cycle. Let one cycle be made up of vertices $x_1, x_2, \ldots, x_r.$ We first take all the elements of the negative adjacency lists of vertices $x_2, \ldots, x_r$ and add them to the negative adjacency list for vertex $x_1.$ Each addition takes constant time, and there are at most $E^- - 2$ additions. Next, we delete the empty lists for vertices $x_2, \ldots, x_r.$ Each deletion again takes constant time, and there are at most $V-2$ deletions. We then look through all the negative adjacency lists to change every appearance of $x_2, \ldots, x_r$ to $x_1.$ Each change takes constant time, and we must look through $2E^-$ elements. Next we look through the adjacency lists and delete every repeated element in a list to get rid of multiple edges. We also delete the negative edge $x_1 \bar x_1$ if it exists. The deletions take constant time, and we must look through $2E^-$ elements. Finally, we concatenate the lists for $x_2, \ldots, x_r$ to the recovery list for $x_1.$ Each addition takes constant time and there are at most $V-1$ lists to append. Therefore, identifying the vertices of one cycle takes time $\mathcal O(V + E^-).$ We do the identification process twice, after completing the depth first search; thus, the time complexity of Step 12 is $\mathcal O(V^2 + E^-).$

\subsubsection{Proof of Theorem \ref{time complexity}}
\begin{proof}
Let $\Sigma$ be a 3-chromatic graph. The algorithm completes Steps 1 and 2 exactly once each. This gives a time complexity of $\mathcal O(V(E^-)^2 + V + E^+).$

Every pass through Steps 3--9 consists of Step 3, at worst checks on all the conditions of Steps 4--9, and the completion of at most one step. Step 3 takes time $\mathcal O(V^2).$ Checking for vertices that satisfy the conditions of Steps 4--9 is bounded by $\mathcal O((E^-)^2).$ The completion of a single step from the Steps 4--9 is bounded by $\mathcal O(V + (E^-)^2).$ Each time the algorithm returns to Step 3, the graph has been reduced by at least one matched pair of vertices. Thus, the algorithm must pass through Steps 3--9 at most $V/2$ times. Therefore, Steps 3--9 are of total time complexity $\mathcal O(V^3 + V(E^-)^2).$

Steps 10, 11, and 12 are repeated at most $V/6$ times. This happens if every pair of cycles the algorithm finds in the forcing graph are composed of 3 vertices each. The cycles cannot be smaller than 3 vertices since Step 7 removes all possible cycles of length 2, and there are no loops in our forcing graphs. Thus, Steps 10, 11, and 12 have total time complexity $\mathcal O(V^3 + VE^-).$

Therefore, the algorithm takes time 
$$\mathcal O(V^3 + V(E^-)^2 + V + E^+) = \mathcal O(V^3 + VE^2 + V + E) = \mathcal O(V^5). \qedhere$$
\end{proof}

\begin{corollary}
MaxDef is a polynomial-time algorithm if the signed graph is without multiple edges of the same sign and is input as a vertex list and two edge lists.
\end{corollary}

\begin{proof}
The vertex list is of size $V$ and the two edge lists are of size $E^+ + E^- = E.$ Thus, by Theorem \ref{time complexity}, the time for MaxDef is bounded by a polynomial in the size of our input.
\end{proof}

\begin{rmk}
A more compact way to input the graph is to input the number of vertices and an edge list. This input is size $E + 1$. If the graph is input in this more compact way, the MaxDef bound of $\mathcal O(V^5)$ is not necessarily polynomial. For example, if $E$ was of size $o(V)$, $\mathcal O(V^5)$ could be exponential or worse. Fortunately, if $\Sigma$ is connected, then $V = \mathcal O(E)$, meaning MaxDef is still a polynomial-time algorithm in the more compact input format.
\end{rmk}

\section{Switching Deficiency}
Let $A$ be a set of vertices of $\Sigma.$ \emph{Switching} $A$ is negating the signs of the edges with exactly one endpoint in $A$. Two signed graphs are \emph{switching equivalent} if they are related by switching. The chromatic number of a signed graph is the same as the chromatic number of the switched graph. A minimal coloration of the switched graph is simply a byproduct of switching the graph; given a minimal coloration of $\Sigma$, the signs of the colors on $A$ are negated during the switching process. Two colorations, $\kappa$ and $\kappa^*,$ are \emph{switching equivalent} if they are related by switching. If $\kappa^*$ is switching equivalent to $\kappa$, then $\kappa^*$ colors $\Sigma^*.$

In general, allowing switching makes the questions surrounding deficiency straightforward to solve. The \emph{switching deficiency range of $\kappa$}, $\Ls(\kappa)$, is $\{\defi(\kappa^*) \mid \kappa^* $ is switching equivalent to $\kappa \}.$ The \emph{switching deficiency range of $\Sigma$}, $\Ls(\Sigma)$, is $\{\defi(\kappa^*) \mid \kappa^* $ is a minimal coloration of a graph that is switching equivalent to $\Sigma.\}.$

\begin{thm}
Let $\Sigma$ be a signed graph and let $\kappa$ be a minimal coloration of $\Sigma.$ Then $\Ls(\kappa) = \Ls(\Sigma) = \left[0, \floor{\frac{1}{2}\chi(\Sigma)}\right].$
\label{switching def main}
\end{thm}

\begin{proof}
Let $\Sigma$ be a signed graph and let $\kappa$ be a minimal coloration of $\Sigma$. 

First we show that $\Ms(\Sigma) = \floor{\frac{1}{2}\chi}.$ We may assume that $\D(\kappa)$ contains only negative colors. Define $\kappa'$ to be the coloration obtained by switching all vertices with color less than 0. Then $\defi(\kappa') = \floor{ \frac{1}{2}\chi}$ and thus, $\Ms(\Sigma) = \floor{\frac{1}{2}\chi}$.

Next we show that $\ms(\Sigma) = 0.$ 

Let $\Sigma$ be a signed graph and let $\kappa$ be a minimal coloration of $\Sigma.$ We prove that for each $i \in D(\kappa)$ there exist at least two vertices colored $-i.$

Suppose $\chi(\Sigma) = 2k$ for some non-negative integer $k$. Suppose there exists an $i \in \D(\kappa)$ such that no negative edge of $\Sigma$ has both endpoints colored $-i$. Then $\{v \mid \kappa(v) = -i\}$ is stable. Thus, recoloring all such vertices 0 results in a proper coloration. But this coloration would use $2k-1$ colors.

Now suppose $\chi(\Sigma) = 2k + 1$ for some non-negative integer $k$. First note that $0 \not \in \D(\kappa)$ and there must at least one vertex colored $-i$ for each $i \in \D(\kappa)$, otherwise $\chi$ would be even. Suppose there exists $i \in \D(\kappa)$ such that exactly one vertex of $\Sigma$ is colored $-i.$ Call this vertex $w.$ Let $\kappa'$ be the coloration defined by 
$$\kappa'(v) = \begin{cases}
i & \text{ if } \kappa(v) = 0 \text{ and } v \text{ is positively adjacent to } w, \\
-i & \text{ if } \kappa(v) = 0 \text{ and } v \text{ is negatively adjacent to } w, \\
-i & \text{ if } \kappa(v) = 0 \text{ and } v \text{ is not adjacent to } w, \\
\kappa(v) & \text{ otherwise.}
\end{cases}$$

The 0-color set of $\kappa$ is stable, so recoloring these vertices $i$ and $-i$ is proper. Also, every vertex that was recolored $i$ is a positive neighbor of $w$. Finally, every vertex that was recolored $-i$ is either negatively adjacent to $w$ or not adjacent to $w$. But then $\kappa'$ is a proper coloration using $2k$ colors.

Switch $\Sigma$ at exactly one of the vertices colored $-i$ for each $i \in D(\kappa),$
and call this new coloration $\kappa'$. Then every color in the color set is being used, so $\defi(\kappa') = 0$ and therefore, $\ms(\Sigma) = 0.$

Finally, we show that every switching deficiency between 0 and $\floor{\frac{1}{2}\chi}$ can be attained. We may assume $\defi(\kappa) = 0.$ Let $r \in [0, \Ms(\kappa)].$ Choose $r$ positive colors in the color set of $\kappa;$ call them $c_1, c_2, \ldots, c_r.$ Let $\kappa'$ be the coloration defined by switching the vertices colored $c_i$ for each $i \in [1, r].$ Then $\kappa'$ does not have any vertices colored $c_i$ for each $i \in [1,r].$ Thus $\defi(\kappa') = \defi(\kappa) + r = r.$ Since every integer in $\Ls(\kappa)$ is an achievable value of switching deficiency for every minimal $\kappa,$ every integer in $\Ls(\Sigma)$ is achievable as well.
\end{proof}

\end{document}